\documentclass[oupthm]{my}

\usepackage{graphicx}
\usepackage{stackengine}
\usepackage{scalerel}
\usepackage{tensor}       
\usepackage{array}
\usepackage{makecell}
\newcolumntype{x}[1]{>{\centering\arraybackslash}p{#1}}
\usepackage{tikz}
\usepackage{pgfplots}

\stackMath

\makeatletter

\tikzset{meter/.append style={draw, inner sep=10, rectangle, font=\vphantom{A}, minimum width=30, line width=.8, path picture={\draw[black] ([shift={(.1,.3)}]path picture bounding box.south west) to[bend left=50] ([shift={(-.1,.3)}]path picture bounding box.south east);\draw[black,-latex] ([shift={(0,.1)}]path picture bounding box.south) -- ([shift={(.3,-.1)}]path picture bounding box.north);}}}

\definecolor{Blues5seq1}{RGB}{239,243,255}
\definecolor{Blues5seq2}{RGB}{189,215,231}
\definecolor{Blues5seq3}{RGB}{107,174,214}
\definecolor{Blues5seq4}{RGB}{49,130,189}
\definecolor{Blues5seq5}{RGB}{8,81,156}

\definecolor{Greens5seq1}{RGB}{237,248,233}
\definecolor{Greens5seq2}{RGB}{186,228,179}
\definecolor{Greens5seq3}{RGB}{116,196,118}
\definecolor{Greens5seq4}{RGB}{49,163,84}
\definecolor{Greens5seq5}{RGB}{0,109,44}

\definecolor{Reds5seq1}{RGB}{254,229,217}
\definecolor{Reds5seq2}{RGB}{252,174,145}
\definecolor{Reds5seq3}{RGB}{251,106,74}
\definecolor{Reds5seq4}{RGB}{222,45,38}
\definecolor{Reds5seq5}{RGB}{165,15,21}

\usepackage[most,breakable]{tcolorbox}
\makeatletter
\renewenvironment{boxed}{\begingroup\@ifnextchar\bgroup\boxed@gobegin\boxed@gobegin@empty}{\end{tcolorbox}\endgroup}
\def\boxed@gobegin#1{\def\@tempa{#1}\def\@tempb{orange}\ifx\@tempa\@tempb\begin{tcolorbox}[colback=red!15,colframe=orange!70,breakable,enhanced]\else\begin{tcolorbox}[colback=Blues5seq1,colframe=Blues5seq5,breakable,enhanced]\fi}
\def\boxed@gobegin@empty{\begin{tcolorbox}[colback=Blues5seq1,colframe=Blues5seq5,breakable,enhanced]}
\makeatother

\usepackage{xcolor}
\usepackage[normalem]{ulem}
\usepackage{dsfont}
\usepackage{cite}
\usepackage[numbers, sort&compress]{natbib}
\usepackage{footnote}
\usepackage{fixfoot}
\usepackage{dsfont}
\usepackage{bigints}
\usepackage{tikz}
\usepackage{geometry}
\usepackage{enumitem}   
\newgeometry{bottom=20mm} 

\usepackage{graphicx}
\usepackage{multicol}
\usepackage{amsmath,amssymb,amsfonts, mathtools}
\usepackage{mathrsfs}

\usepackage{rotating}
\usepackage[title, toc]{appendix}
\usepackage[numbers]{natbib}
\usepackage{changepage,afterpage}
\usepackage{etruscan}
\usepackage{tikz, siunitx, newunicodechar}
\usepackage{booktabs,multirow}
\tcbuselibrary{skins,breakable}
\usepackage{xcolor}
\usepackage{lipsum}
\usepackage{longtable}

\usepackage{bm,bbm}
\usepackage{braket}

\usepackage[utf8]{inputenc}
\usepackage{array}
\usepackage{colortbl}

\allowdisplaybreaks[4]

\usepackage[T3,T1]{fontenc}
\DeclareSymbolFont{tipa}{T3}{cmr}{m}{n}
\DeclareMathAccent{\invbreve}{\mathalpha}{tipa}{16}

\theoremstyle{oupplain}
\newtheorem{theorem}{Theorem}[section]
\newtheorem{lemma}[theorem]{Lemma}

\newtheorem{corollary}[theorem]{Corollary}
\theoremstyle{oupdefinition}

\theoremstyle{oupremark}

\newtheorem{example}[theorem]{Example}
\theoremstyle{oupproof}

\definecolor{ashgrey}{rgb}{0.7, 0.75, 0.71}
\newenvironment{proof}{
	\tcolorbox[blanker,breakable,left=5mm,parbox=false,
    before upper={\parindent15pt},
    after skip=10pt,
	borderline west={0.7mm}{0pt}{Blues5seq1}]
    {\noindent{\it \textbf{Proof:}}}
}{
    \textcolor{black}{\hbox{}\nobreak\hfill$\blacksquare$} 
    \endtcolorbox
}

\numberwithin{equation}{section}

\articletype{RESEARCH ARTICLE}

\begin{document}
\setlength{\abovedisplayskip}{12pt}
\setlength{\belowdisplayskip}{12pt}

\begin{Frontmatter}
\title{Weak supermajorization between symplectic spectra of positive definite matrix and its pinching}

    \author{Temjensangba \hspace{-0.10cm}\thanks{Department of Mathematics and Computing, Indian Institute of Technology (ISM) Dhanbad, Jharkhand 826004, India; {\email{temjensangba111@gmail.com}}} and Hemant K. Mishra \hspace{-0.10cm}\thanks{Department of Mathematics and Computing, Indian Institute of Technology (ISM) Dhanbad, Jharkhand 826004, India; {\email{hemantmishra1124@iitism.ac.in}}}
    }

    \abstract 
    { Let $A = \begin{bsmallmatrix} E & F \\ F^T  & G \end{bsmallmatrix}$ be a $2n \times 2n$ real positive definite matrix, where $E, F,$ and $G$ are $n \times n$ blocks.
 It is shown that
 \begin{equation*}
     d(E \oplus G) \prec^w d(A).
 \end{equation*}
  Here $d(A)$ denotes the $n$-vector consisting of the symplectic eigenvalues of $A$ arranged in the non-decreasing order.
    We also observe the following weak supermajorization relation, which is interesting on its own:
    \begin{align*} 
     \lambda \left( \left(\mathscr{C}(G)^{1/2} \mathscr{C}(E) \mathscr{C}(G)^{1/2}\right)^{1/2} \right) \prec^w \lambda \left( \left(G^{1/2} E G^{1/2} \right)^{1/2} \right).
 \end{align*} 
   Here $\lambda \left( \left( G^{1/2}E G^{1/2} \right)^{1/2} \right)$ denotes the $n$-vector with entries given by the eigenvalues of $\left( G^{1/2}E G^{1/2} \right)^{1/2}$ in the non-decreasing order.
  }

\end{Frontmatter}
    
\section{Introduction}
Symplectic eigenvalues have been extensively explored over the last two decades and it continues to be a vibrant research field.
Several analogs of classic eigenvalue results are developed for symplectic eigenvalues \cite{Mishra2020_first, BhatiaJai2021_variational, Jain2021_sumsandproducts, JainMishra2022_derivatives, Paradan2022_theHorncone, son2022symplectic, BabuMishra2024_blockperturbation, kamat2026simultaneous}.
In particular, symplectic versions of several majorization results on eigenvalues, such as Lidskii's theorem \cite{JainMishra2022_derivatives} and Schur--Horn theorem \cite{BhatiaJain2020_schurhorn, huang2023newversion} are known today, and these results work with weak supermajorization in place of majorization.
For instance, it is known that the eigenvalues of a Hermitian matrix majorize the eigenvalues of its pinching (see, e.g., \cite{lin2012_eigenvaluemajorization}).
A symplectic analog of this result given in \cite[Theorem~9]{BhatiaJain2015_onsymplectic} states that the symplectic eigenvalues of a positive definite matrix weakly supermajorize the symplectic eigenvalues of its \emph{symplectic pinching}.
It is to be noted that a symplectic pinching operation is fundamentally different from a classical pinching operation.
To the best of our knowledge, no symplectic analog of the aforementioned eigenvalue result is known for a classical pinching.

In this letter, we establish a weak supermajorization relation between the symplectic eigenvalues of a real positive definite matrix and that of its \emph{pinching} leaving two diagonal blocks of the same size.
More concretely, let
$A = \begin{bsmallmatrix}
    E & F \\ F^T & G
\end{bsmallmatrix}$
be a $2n \times 2n$ real positive definite matrix, where $E, F$, and $G$ are $n \times n$ matrices.
In our main result, we show that the symplectic eigenvalues of $E \oplus G$ are weakly supermajorized by the symplectic eigenvalues of $A$.
We also report some interesting consequences of the main result.

The rest of the paper is structured as follows. 
In Section~\ref{preliminaries}, we briefly discuss some prerequisites.
The main result and its consequences are given in Section~\ref{main result}.

\section{Preliminaries} \label{preliminaries}
In this section, we briefly discuss the prerequisites needed for our discourse.
Denote by $\mathds{M}_{m,n}$ the set of all $m \times n$ real matrices and let $\mathds{M}_n \coloneqq \mathds{M}_{n, n}$.
Let $\mathds{P}_{n} \subset \mathds{M}_{n}$ consisting of the positive definite matrices.
For $A \in \mathds{M}_n$, let $\sigma(A)$ denote the eigen spectrum of $A$. If all the eigenvalues of $A$ are real, let $\lambda(A) = [\lambda_1(A), \ldots, \lambda_n(A)]^T$ denote the $n$-vector whose entries are the usual eigenvalues of $A$ arranged in the non-decreasing order. For any $A = [a_{ij}] \in \mathds{M}_n$, we write $\Delta(A) = \left[ a_{11}, \ldots, a_{nn} \right]^T$.

A matrix $M \in \mathds{M}_n$ is said to be symplectic if it satisfies
\begin{equation} \label{eq:definition of symplectic  matrix}
    M^T J_{2n} M = J_{2n},
\end{equation}
where $J_{2n} \coloneqq \begin{bsmallmatrix}
        0 & I_n \\ -I_n & 0
    \end{bsmallmatrix}$, 
and $I_n$ denotes the identity matrix of size $n \times n$.
We shall drop the ``$2n$'' from $J_{2n}$ wherever it is clear from the context.
The set of all $2n \times 2n$ symplectic matrices forms a group under matrix multiplication called the symplectic group, and we denote this group by $\mathds{SP}_{2n}$.
A set of vectors $\{ x_1,\ldots x_n, y_1,\ldots, y_n \}$ in $\mathds{R}^{2n}$ is said to be a symplectic basis if it satisfies for each $1 \leq k, \ell \leq n$:
\begin{align}
    x_k^T J x_\ell &= y_k^T J y_\ell = 0, \\
    x_k^T J y_\ell &= \begin{cases}
        1, \text{ if } k = \ell \\
        0, \text{ if } k \neq \ell.
    \end{cases}
\end{align}
There is a one-to-one correspondence between the set of all symplectic bases of $\mathds{R}^{2n}$ and the symplectic group $\mathds{SP}_{2n}$.

Williamson's theorem \cite{Williamson} states that for any $A \in \mathds{P}_{2n}$, there exists $M \in \mathds{SP}_{2n}$ such that
\begin{equation} \label{eq:williamson's decomposition}
    M^T A M = D \oplus D,
\end{equation}
where $D \in \mathds{P}_{n}$ is a diagonal matrix which is unique up to permutation of its diagonal entries.
The diagonal entries of $D$ are called the symplectic eigenvalues of $A$.
Let $\sigma_s(A)$ denote the set of all symplectic eigenvalues of $A$ and call it the symplectic spectrum of $A$.
Let $d(A)=[d_1(A),\ldots, d_n(A)]^T$ denote the $n$-vector consisting of the symplectic eigenvalues of $A$ with entries arranged in the non-decreasing order.
The set of all symplectic matrices which diagonalizes $A$ in the sense of Williamson's theorem \eqref{eq:williamson's decomposition} is denoted by $\mathds{SP}_{2n}(A)$.
If $\{ x_1,\ldots, x_n, y_1,\ldots, y_n \}$ is the symplectic basis of $\mathds{R}^{2n}$ given by the columns of $M$ in \eqref{eq:williamson's decomposition}, then we have 
\begin{equation}
    Ax_k = d_k(A) Jy_k, \quad Ay_k = -d_k(A) J x_k, \qquad \forall 1 \leq k \leq n.
\end{equation}
We call each pair $(x_k, y_k)$ a symplectic eigenvector pair of $A$ corresponding to the symplectic eigenvalue $d_k(A)$.
It follows from \eqref{eq:definition of symplectic  matrix} and \eqref{eq:williamson's decomposition} that $M \in \mathds{SP}_{2n}(A)$  if and only if the $k$th and the $(n+k)$th columns of $M$ form a symplectic eigenvector pair of $A$ corresponding to the symplectic eigenvalue $d_k(A)$ for all $k \in \{1, \ldots, n \}$.

Suppose $m_1, \ldots, m_k \in \mathds{N}$ and $n=m_1 + \ldots + m_k$.
Let $H=[H_{ij}] \in \mathds{M}_n$ with blocks $H_{ij} \in \mathds{M}_{m_i,m_j}$ for $1 \leq i,j \leq k$.
A pinching of $M$ relative to $(m_1, \ldots, m_k)$, denoted by $\mathscr{C}(H)$, is defined as
\begin{equation} \label{eq:definition_of_pinching}
     \mathscr{C}(H) = \bigoplus_{j = 1}^k H_{jj},
\end{equation}
which is the direct sum of $H_{11}, \ldots, H_{kk}$.
Let $A = \begin{bsmallmatrix}
    E & F \\ F^T & G
\end{bsmallmatrix}
\in \mathds{P}_{2n}$, where $E, F, G$ are $n \times n$ blocks.
The symplectic pinching or $s$-pinching of $A$ relative to $(m_1, \ldots, m_k)$, denoted by $\mathscr{C}^s(A)$, is defined as
\begin{equation} \label{eq:definition_of_s_pinching}
    \mathscr{C}^s(A) := \begin{bmatrix}
        \mathscr{C}(E) & \mathscr{C}(F) \\
        \mathscr{C}(F^T) & \mathscr{C}(G)
    \end{bmatrix}.
\end{equation}
Also, define 
\begin{equation} \label{eq:definition_of_symplectic_diagonal}
    \Delta_s(A) \coloneqq \left[ \sqrt{\eta_{1} \gamma_{1}}, \ldots, \sqrt{\eta_{n} \gamma_{n}} \right]^T,
\end{equation}
where $\eta_{i}$ and $\gamma_i$ are the $i$th diagonal entries of $E$ and $G$, respectively, for $1 \leq i \leq n$.

For any $x \in \mathds{R}^n$, we denote the entries of $x$ arranged in non-decreasing order by $x_1^\uparrow, \ldots, x_n^\uparrow$.
For any $x, y \in \mathds{R}^n$,  we say that $x$ is weakly supermajorized by $y$, denoted by $x \prec^w y$, if
\begin{equation} \label{eq:weak_supermajorization_definition}
    \sum_{j = 1}^k x_j^\uparrow \geq \sum_{j = 1}^k y_j^\uparrow, \qquad 1 \leq k \leq n.
\end{equation}
If equality holds in \eqref{eq:weak_supermajorization_definition} for $k = n$, we say that $x$ is majorized by $y$, denoted by $x \prec y$.

\section{Main result} \label{main result}
Throughout this section, we shall consider $A = \begin{bsmallmatrix} E & F \\ F^T  & G \end{bsmallmatrix} \in \mathds{P}_{2n}$, where $E, F, G$ are $n \times n$ blocks.
It will be worthwhile to recall from \cite[Section 2]{BhatiaJain2020_schurhorn} that, if $E, F,$ and $G$ are diagonal matrices, then
\begin{equation} \label{eq:symplectic_eigenvalues_of_diagonal}
    \sigma_s(A) = \left\{ \sqrt{\eta_1 \gamma_1 - \beta_1^2},\ldots, \sqrt{\eta_n \gamma_n - \beta_n^2} \right\},
\end{equation}
where $\eta_j, \gamma_j,$ and $\beta_j$ are the $j$th diagonal entries of $E, G,$ and $F$ for $1 \leq j \leq n$, respectively.

The following lemma will be useful in the proof of the main result.

\begin{boxed}
\begin{lemma} \label{thm:main_lemma}
 We have
 \begin{align} \label{eq:d(b) = lambda(GEG)}
     d(E \oplus G) &= \lambda \left( \left(G^{1/2} E G^{1/2} \right)^{1/2} \right).
 \end{align}
 Consequently, we have
 \begin{equation} \label{eq:square_roots_of_the_eigenvalues}
    \sigma_s(E \oplus G) = \left\{ \sqrt{\lambda} : \lambda \in \sigma(EG) \right\}.
 \end{equation}
 \end{lemma}
 \end{boxed}

 \begin{proof}
     We have
     \begin{equation}
         \left( G^{1/2} \oplus G^{-1/2} \right) \left(E \oplus G \right) \left( G^{1/2} \oplus G^{-1/2} \right) = (G^{1/2} E G^{1/2}) \oplus I_n.
     \end{equation}
By the spectral theorem, there exists an orthogonal matrix $P$ such that $G^{1/2} E G^{1/2}=P^T \Lambda P$, where $\Lambda = \operatorname{diag}(\lambda(G^{1/2} E G^{1/2}))$.
Since $\left( G^{1/2} \oplus G^{-1/2} \right)$ is a symplectic matrix, it follows that 
\begin{align}
    d \left( E \oplus G \right) 
        &= d \left( \left(G^{1/2} E G^{1/2} \right) \oplus I_n \right) \\
        &= d \left( P^T \Lambda P \oplus I_n \right) \\
        &= d \left( \left( P \oplus P \right)^T
         \left( \Lambda \oplus I_n \right)
         \left( P \oplus P \right) \right) \\
        &= d \left( \Lambda \oplus I_n \right) \\
        &= \left[ \sqrt{\lambda_1 \left(G^{1/2} E G^{1/2} \right)}, \ldots, \sqrt{\lambda_n \left( G^{1/2} E G^{1/2} \right)} \right]^T \label{eq:lambda oplus identity} \\
        &= \left[ \lambda_1 \left( \left( G^{1/2} E G^{1/2} \right)^{1/2} \right), \ldots, \lambda_n \left( \left(G ^{1/2} E G^{1/2} \right)^{1/2} \right) \right]^T,
\end{align}
where \eqref{eq:lambda oplus identity} follows from \eqref{eq:symplectic_eigenvalues_of_diagonal}.
This proves \eqref{eq:d(b) = lambda(GEG)}. 
Since $EG = G^{-1/2} \left( G^{1/2} E G^{1/2} \right)G^{1/2}$, we have
\begin{equation} \label{eq:similarity of matrices}
    \sigma \left( EG \right) = \sigma \left( G^{1/2} E G^{1/2} \right),
\end{equation}
which implies that
\begin{equation} \label{eq:this_proves_that}
    \left\{ \sqrt{\lambda} : \lambda \in \sigma(EG) \right\} = \sigma \left( \left( G^{1/2} E G^{1/2} \right)^{1/2} \right).
\end{equation}
The relations \eqref{eq:d(b) = lambda(GEG)} and \eqref{eq:this_proves_that} imply \eqref{eq:square_roots_of_the_eigenvalues}.
 \end{proof}

The following theorem is the main result.
\begin{boxed}
\begin{theorem} \label{thm:weak_supermajorization_of_pinching}
    We have
    \begin{equation} \label{eq:pinching-maj-symplectic}
        d(E \oplus G) \prec^w d(A).
    \end{equation}
\end{theorem}
\end{boxed}

\begin{proof}
    We have seen in \eqref{eq:similarity of matrices} that
\begin{equation}
    \sigma \left( EG \right) = \sigma \left( G^{1/2} E G^{1/2} \right).
\end{equation}
By the spectral theorem, there exists an orthonormal basis $\{v_1, \ldots, v_n\}$ of $\mathds{R}^n$ such that
\begin{equation} \label{eq:orthonormal_eigenvectors}
    G^{1/2} E G^{1/2} v_k = \lambda_k(EG) v_k, \qquad \forall 1 \leq k \leq n.
\end{equation}
Now, define $u_k \coloneqq \sqrt{d_k(E \oplus G)} G^{-1/2} v_k$.
Observe that
\begin{align}
    E G u_k &=  E G \left( \sqrt{d_k(E \oplus G)} G^{-1/2} v_k \right) \\
    &= \sqrt{d_k(E \oplus G)}  E G^{1/2} v_k \\
    &= \sqrt{d_k(E \oplus G)}  G^{-1/2} \left( G^{1/2} E G^{1/2} v_k  \right) \\
    &= \lambda_k(EG) \left( \sqrt{d_k(E \oplus G)} G^{-1/2} v_k \right) \\
    &= \lambda_k(EG) u_k. \label{eq:ge-eigenvalue-eqn}
\end{align}
By substituting $\lambda_k(EG) = d_k^2(E \oplus G)$ from \eqref{eq:square_roots_of_the_eigenvalues} into \eqref{eq:ge-eigenvalue-eqn}, we get
\begin{equation} \label{eq:uj_is_still_an_eigenvector_of_GE}
    EG u_k = d_k^2(E \oplus G) u_k, \qquad \forall 1 \leq k \leq n.
\end{equation}
Moreover, it can be easily verified that for every $ k, \ell \in \{1,\ldots, n \}$,
\begin{equation} \label{eq:kronecker_delta_relationship}
    u_k^T G u_\ell = \sqrt{d_k(E \oplus G)} \sqrt{d_\ell(E \oplus G)} \delta_{k\ell},
\end{equation}
where $\delta_{k\ell}=0$ if $k \neq \ell$, and $\delta_{kk}=1$.
Define
\begin{equation} \label{eq:define_xj_and_yj}
    x_k \coloneqq \begin{bmatrix} 0 \\ u_k \end{bmatrix}, \quad
    y_k \coloneqq \begin{bmatrix}  \frac{-1}{d_k(E \oplus G)}  G u_k \\ 0  \end{bmatrix} \qquad \forall 1 \leq k \leq n.
\end{equation}
It can be verified using \eqref{eq:kronecker_delta_relationship} and \eqref{eq:define_xj_and_yj} that
$\{x_1,\ldots, x_n, y_1,\ldots, y_n\}$ forms a symplectic basis of $\mathds{R}^{2n}$ so that $M \coloneqq \left[ x_1, \ldots, x_n, y_1, \ldots, y_n \right] \in \mathds{SP}_{2n}$.
In fact, we have $M \in \mathds{SP}_{2n}(E \oplus G)$.
Indeed, by using the block forms of $x_k, y_k$ given by \eqref{eq:define_xj_and_yj}, we have for each $k \in \{1, \ldots, n \}$,
\begin{align}\label{eq:symplectic_equation}
    (E \oplus G)x_k &= d_k(E \oplus G) J y_k, \\
    (E \oplus G)y_k &= -d_k(E \oplus G) J x_k.
\end{align}

Now, let $k \in \{1,\ldots, n\}$ be arbitrary and define $W \coloneqq \left[ x_1, \ldots x_k, y_1, \ldots, y_k \right]$. 
It is straightforward to verify that $W^T J_{2n} W = J_{2k}$. 
Also, by the structure of $W$ and a routine computation with the trace operation, we get $\operatorname{tr}\! \left( W^T (E \oplus G) W \right) = \operatorname{tr}\! \left( W^T A W \right)$.
This gives
\begin{align}
 2 \sum_{j = 1}^k d_j(E \oplus G) &= \operatorname{tr}\! \left( W^T \left( E \oplus G \right) W \right) \\
 &= \operatorname{tr}\! \left( W^T A W \right) \\
 &\geq \min_{\substack{X \in \mathds{M}_{2n, 2k} \\ X^T J_{2n} X = J_{2k}}} \operatorname{tr}\! \left( X^T A X \right) \label{eq:trace_XAX} \\
 &= 2 \sum_{j=1}^k d_j(A). \label{eq:symplectic_ky_fan_appl}
\end{align}
The last equality is due to \cite[Theorem~5]{BhatiaJain2015_onsymplectic}.
This proves that $d(E \oplus G) \prec^w d(A)$.
\end{proof}

For the symplectic pinching $\mathscr{C}^s$ relative to $(1,\ldots, 1)\in \mathds{R}^n$ so that $\mathscr{C}^s (E \oplus G) = \Delta(E) \oplus \Delta(G)$, we have $d( \mathscr{C}^s (E \oplus G)) = \Delta_s(A)$ by Lemma~\ref{thm:main_lemma}, where $\Delta_s(A)$ is defined in \eqref{eq:definition_of_symplectic_diagonal}. 
It thus follows from Theorem~9 of \cite{BhatiaJain2015_onsymplectic} that
\begin{equation} \label{eq:symp-dia-pinching-wsm}
    \Delta_s(A) \prec^w d(E \oplus G).
\end{equation}
The relations \eqref{eq:pinching-maj-symplectic} and \eqref{eq:symp-dia-pinching-wsm} together imply that
\begin{equation} \label{eq:symplectic_Schur's_theorem}
    \Delta_s(A) \prec^w d(A).
\end{equation}
The weak supermajorization relation in \eqref{eq:symplectic_Schur's_theorem} is a symplectic analog of the classic Schur's theorem established in \cite[Theorem~1]{BhatiaJain2020_schurhorn}.
Furthermore, for any pinching $\mathscr{C}$, we appeal to Theorem~9 of \cite{BhatiaJain2015_onsymplectic} again to write
\begin{equation} \label{eq:d(C(E)+C(G)) is supermajorized by d(E+G)}
    d \left( \mathscr{C}(E) \oplus \mathscr{C}(G) \right) \prec^w d \left( E \oplus G \right).
\end{equation}
Our result \eqref{eq:pinching-maj-symplectic} together with \eqref{eq:d(C(E)+C(G)) is supermajorized by d(E+G)} yields
\begin{equation}
  d \left( \mathscr{C}(E) \oplus \mathscr{C}(G) \right) \prec^w d(A).
\end{equation}

The following example is a striking illustration of the fact that the weak supermajorization relation in \eqref{eq:pinching-maj-symplectic} may fail to hold for other types of pinching of $A$.

\begin{boxed}
    \begin{example}
        Let $A = \begin{bsmallmatrix}
            7 & 6 \\
            6 & 7
        \end{bsmallmatrix} \oplus 
        \begin{bsmallmatrix}
            7 & 6 \\
            6 & 7
        \end{bsmallmatrix}$
    and consider the pinching
    $\mathscr{C}(A) = \begin{bsmallmatrix}
            7 & 0 \\
            0 & 7
        \end{bsmallmatrix} \oplus 
        \begin{bsmallmatrix}
            7 & 6 \\
            6 & 7
        \end{bsmallmatrix}$
    of $A$ relative to $(1,3)$.
    Computation yields $d(A) = [1, 13]^T$ and $d(\mathscr{C}(A)) = [ 2.65, 9.54]^T$ showing that $d(\mathscr{C}(A))$ is not weakly supermajorized by $d(A)$.
    \end{example}
\end{boxed}

We now turn to discussing a few consequences of our main result.

\begin{boxed}
\begin{corollary} \label{thm:the_first_corollary}
 The eigenvalues of $\left(G^{1/2} E G^{1/2}\right)^{1/2}$ are weakly supermajorized by the symplectic eigenvalues of $A$, i.e.,
\begin{equation} \label{eq:the_first_consequence}
 \lambda \left( \left(G^{1/2} E G^{1/2}\right)^{1/2} \right) \prec^w d(A).   
 \end{equation}
\end{corollary}
\end{boxed}
\begin{proof}
 By substituting \eqref{eq:d(b) = lambda(GEG)} into \eqref{eq:pinching-maj-symplectic}, we get
\eqref{eq:the_first_consequence}.
 \end{proof}
 It is known \cite{bhatia1990_onthesingularvalues} that
 \begin{equation} \label{eq:Bhatia's singular value inequality}
     \lambda \left( \left( G^{1/2} E G^{1/2} \right)^{1/2} \right) \leq \lambda \left( \frac{E + G}{2} \right).
 \end{equation}
 The above inequality together with \eqref{eq:the_first_consequence} implies that
 \begin{equation} \label{eq:the second consequence}
     \lambda \left( \frac{E + G}{2} \right) \prec^w d(A).
 \end{equation}
 Alternatively, we can also arrive at \eqref{eq:the second consequence} by the following arguments.
Observe that $d(E \oplus G) = d(G \oplus E)$.
We have
\begin{align}
    \lambda\left( E+G \right) 
    &=d \left( \left( E + G \right) \oplus d \left( E + G \right) \right) \label{eq:from our main lemma}  \\
    &= d \left( (E \oplus G) + (G \oplus E) \right) \\
    &\prec^w d \left( E \oplus G \right) + d \left( G \oplus E \right) \label{eq:Hiroshima's result} \\
    &= 2 d(E \oplus G) \\
    &\prec^w 2 d(A). \label{eq:from our main result}
\end{align}
The equality \eqref{eq:from our main lemma} follows from \eqref{eq:square_roots_of_the_eigenvalues}, the weak supermajorization \eqref{eq:Hiroshima's result} is due to Theorem~1 of \cite{hiroshima2006_additivity}, and
\eqref{eq:from our main result} is from Theorem~\ref{thm:weak_supermajorization_of_pinching}.

\begin{boxed}
    \begin{corollary}
       For any pinching $\mathscr{C}$, we have
 \begin{align} \label{eq:the_third_consequence}
     \lambda \left( \left(\mathscr{C}(G)^{1/2} \mathscr{C}(E) \mathscr{C}(G)^{1/2} \right)^{1/2} \right) \prec^w \lambda \left( \left(G^{1/2} E G^{1/2}\right)^{1/2} \right).
 \end{align} 
In particular, 
 \begin{equation} \label{eq:G = I}
 \lambda \left( \mathscr{C}(E)^{1/2} \right) \prec^w \lambda \left(E^{1/2}\right).
 \end{equation}
    \end{corollary}
\end{boxed}

\begin{proof}
By Lemma~\ref{thm:main_lemma}, we have
\begin{align}
d \left( E \oplus G \right) &= \lambda \left( \left(G^{1/2} E G^{1/2}\right)^{1/2} \right), 
\end{align}
and
\begin{align}
    d \left( \mathscr{C}(E) \oplus \mathscr{C}(G) \right)
    &= \lambda \left( \left( \mathscr{C}(G)^{1/2} \mathscr{C}(E) \mathscr{C}(G)^{1/2} \right)^{1/2} \right).
\end{align}
    The weak supermajorization relation \eqref{eq:the_third_consequence} then follows by substituting the above identities into \eqref{eq:d(C(E)+C(G)) is supermajorized by d(E+G)}.
    Also, \eqref{eq:G = I} follows by choosing $G$ equal to the identity matrix in \eqref{eq:the_third_consequence}.
\end{proof}

An alternative way to attain \eqref{eq:G = I} is as follows.
For any Hermitian matrix $E$ and any pinching $\mathscr{C}$, we have discussed that
\begin{equation} \label{eq:final result of pinching}
    \lambda \left( \mathscr{C}(E) \right) \prec \lambda(E).
\end{equation}
Noting that $\sqrt{x} : [0, \infty) \longrightarrow [0, \infty)$ is a concave function, an appeal to Theorem~5.A.1 of \cite{marshall1979inequalities} gives \eqref{eq:G = I}.

\section*{Acknowledgements}
    The first author thanks Nagaland University for granting study leave with pay. 
    The second author acknowledges support from FRS Project No.~MISC~0147.

\begin{backmatter}

\bibliographystyle{unsrtnat}
\bibliography{reference}
\end{backmatter}
\printaddress


    
\end{document}